\documentclass[11pt]{article}

\usepackage{amsmath,amssymb,amsfonts,dsfont}

\usepackage{enumerate}
\usepackage{graphicx,psfrag}
\usepackage{tikz}
\graphicspath{ {./figure/} }

\usepackage{amsthm}
\newtheorem{theorem}{Theorem} 
\newtheorem{lemma}[theorem]{Lemma} 

\newtheorem{conj}[theorem]{Conjecture}

\newcommand{\R}{\mathbb{R}}

\newcommand{\N}{\mathbb{N}}
\newcommand{\Z}{\mathbb{Z}}

\newcommand{\0}{\mathbf{0}}
\newcommand{\xx}{\mathbf{x}}

\newcommand{\hh}{\mathbf{h}}

\newcommand{\ud}{\,\mathrm{d}}

\newcommand{\kcost}{c} 

\newcommand{\Reff}{R_{\rm eff}}
\newcommand{\Rave}{R_{\rm ave}}
\newcommand{\Raveoint}{\subscr{\mathring{R}}{ave}}
\newcommand{\Raveoext}{\subscr{\overline{R}}{ave}}
\newcommand{\Raveb}[1]{\Rave \l(#1\r)}
\newcommand{\Ravebo}[1]{\subscr{\mathring{R}}{ave}\l(#1\r)}

\newcommand{\Ravehydro}[1]{ \gamma\l(#1\r) }

\newcommand{\TMd}{T_{M^d}}
\newcommand{\TMm}{T_{M^m}}
\newcommand{\Muno}{{M_{1}}}
\newcommand{\Mdue}{{M_{2}}}
\newcommand{\TMuno}{T_\Muno}
\newcommand{\TLuno}{T_\Mdue}
\newcommand{\TML}{T_{\Muno,\Mdue}} 

\newcommand{\TM}{T_M}
\newcommand{\TMdue}{ T_{M^2} }
\newcommand{\TMtre}{ T_{M^3}  }
\newcommand{\TMquattro}{ T_{M^4}  }

\newcommand{\Hyp}[1]{H_{#1}}

\newcommand{\AAA}{A}
\newcommand{\BBB}{B}
\newcommand{\CCC}{C}
\newcommand{\Ok}{ \mathcal{I}^A_d}
\newcommand{\BZk}{\mathcal{I}^B_d}
\newcommand{\Sk}{ \mathcal{I}^C_d}

    \newcommand{\fromto}[2]{\{#1,\dots, #2\}}
    
    \newcommand{\subscr}[2]{#1_{\textup{#2}}}
    
    \newcommand{\setdef}[2]{\{#1 \, : \; #2\}}
    \newcommand{\map}[3]{#1: #2 \rightarrow #3}

    \renewcommand{\l}{\left}
    \renewcommand{\r}{\right}

\newlength{\figtriml}
\newlength{\figtrimb}
\newlength{\figtrimr}
\newlength{\figtrimt}
\newlength{\figtrimbX}
\newlength{\figtrimtX}
\newlength{\figwidth}
\newlength{\figwidthcap}
\newlength{\figwidthduo}
\title{Average resistance of toroidal graphs%
\thanks{%
The authors wish to thank F.~Garin for fruitful conversations on the topics of this paper.
Their work has been partly supported by the Italian Ministry MIUR under grant PRIN-20087W5P2K.}
}

\author{Wilbert Samuel Rossi%
	\thanks{W. S. Rossi is with Department of Applied Mathematics, University of Twente, 7500 AE Enschede, The Netherlands;
	{\tt\small w.s.rossi@utwente.nl}.}
\and Paolo Frasca%
	\thanks{P. Frasca is with Department of Applied Mathematics, University of Twente, 7500 AE Enschede, The Netherlands;
	{\tt\small p.frasca@utwente.nl}.} 
\and Fabio Fagnani%
	\thanks{F. Fagnani is with Dipartimento di Scienze Matematiche, 
	Politecnico di Torino, corso Duca degli Abruzzi 24, 10129 Torino, Italy; 
	{\tt\small fabio.fagnani@polito.it}.}%
}

\begin{document}
\maketitle

\begin{abstract}
	\noindent  The average effective resistance of a graph is a relevant performance index in many applications, including distributed estimation and control of network systems. In this paper, we study how the average resistance depends on the graph topology and specifically on the dimension of the graph. We concentrate on $d$-dimensional toroidal grids and we exploit the connection between resistance and Laplacian eigenvalues. Our analysis provides tight estimates of the average resistance, which are key to study its asymptotic behavior when the number of nodes grows to infinity. In dimension two, the average resistance diverges: in this case, we are able to capture its rate of growth when the sides of the grid grow at different rates. In higher dimensions, the average resistance is bounded uniformly in the number of nodes: in this case, we conjecture that its value is of order $1/d$ for large $d$. We prove this fact for  hypercubes and when the side lengths go to infinity.
\end{abstract}


\section{Introduction}  \label{sec:intro}

	The effective resistance between nodes of a graph is a classical fundamental concept that naturally comes up when 
	the graph is interpreted as an electrical network. For several decades, it has been known 
	to play a key role in the theory of  time-reversible Markov chains, because of its connections with 
	escape probabilities and commute times~\cite{PGD-JLS:84, AKC-PR89, AF:94,LPW:08}. 
	More generally, the notion of effective resistance has broad application in science: 
	%
	%
	in chemistry, for instance, the total effective resistance (summed over all pairs of nodes) 
	is known as the Kirchhoff index of the graph, where the graph of interest has the atoms as 
	nodes and their bonds as edges. 
	This classical index is linked to the properties of organic macromolecules~\cite{DB-EJM-AHD:02} 
	and to the vibrational energy of the atoms: the latter property has also been interpreted as a 
	measure of vulnerability in complex networks~\cite{EE-NH:10}.

\subsection*{Effective resistance in network systems}
	Recently, the average effective resistance of a graph has appeared as an important performance index in several 
	network-oriented problems of control and estimation, where the nodes (or agents) collectively 
	need to obtain estimates of given quantities with limited communication effort. 
	One instance is the {\em consensus} problem, where a set of agents, each with a scalar value, has the goal is to reach a common state that is a weighted average of the initial values. This problem can be solved by a simple linear iterative algorithm, which has 
	become very popular. The performance of this algorithm depends on the graph representing the 
 communication between the agents and the average effective resistance 
	of this graph plays a key role~\cite{RC-FG-SZ:09,FG-LS:11,EL-FG-SZ:13}. 
	Indeed, the average resistance determines both the convergence speed during the transient~\cite[Section~3.4]{AG-SB-AS:08}~\cite{FG-SZ:11} and the robustness against additive noise affecting the updates~\cite{LX-SB-SJK:07}: in the latter case,  the effective resistance of the graph is proportional to	 the mean deviation of the states from their average when time goes to infinity.
Similar issues of robustness to disturbances for network systems, such as platooning of vehicles, have attracted much interest~\cite{BB-MJ-PM-SP:12}. 

Another relevant problem is the {\em relative estimation} problem: each node is endowed with a value and these values have to be estimated by using noisy measurements of differences taken along the available edges. The expected error of the least-squares estimator is proportional to the average effective resistance of the graph~\cite{PB-JPH:08}. 
This estimation problem arises in several applications, ranging from clock synchronization~\cite{JE-RMK-CHP-SS:04,AG-PRK:06} to self-localization of mobile robotic networks~\cite{PB-JPH:09} and to statistical ranking from pairwise comparisons~\cite{AL-CDM:12,BO-CB-SO:14}. Several distributed algorithms that solve the relative estimation problem have been recently studied~\cite{PB-JPH:07,WSR-PF-FF:12,WSR-PF-FF:13,RC-AC-LS-MT:13,CR-PF-RT-HI:13c}.

In all the above situations, performance improves when the effective resistance is reduced. This observation motivates, for instance, the problem of allocating edge weights on a the edges of given graph in order to minimize the average effective resistance~\cite{AG-SB-AS:08}. Similarly, it motivates our interest in topologies ensuring small average resistance. More precisely, we consider families of graphs and we ask whether the average resistance depends gracefully on the size.

\subsection*{Effective resistance and graph dimension}

As we have argued, the average effective resistance of a graph is a relevant index in several problems. When one tries to understand the dependence of this index on the topology, it comes out that the notion of {\em dimension} of the graph plays an essential role. 
%
%
	It is well known~\cite{PB-JPH:07,BB-MJ-PM-SP:12} that in grid-like graphs of dimension $d$ and size $N$ 
	(the cardinality of the set of vertices), the average effective resistance $\Rave$ scales%
	\footnote{Given two sequences $\map{f,g}{\N}{\R^+}$, let $\ell^+=\limsup_n f(n)/g(n)$ 
	and $\ell^-=\liminf_n f(n)/g(n)$.
	We write that $f=O(g)$ when $\ell^+<+\infty$; that 
	$f=o(g)$ when $\ell^+=0$; that $f \sim g$ when $\ell^+=\ell^-=1$, and $f=\Theta(g)$ when 
	$\ell^+,\, \ell^-\in (0,+\infty)$. Finally, we write $f=\Omega(g)$ when $g=O(f)$.} 
	in $N\to +\infty$ (and fixed $d$) as follows
	$$\Rave= \left\{\begin{array}{ll} \Theta(N)\;\;&d=1\\ \Theta(\ln N) \;\; &d=2\\ \Theta(1)\;\; &d\geq 3\end{array}\right.$$
Notwithstanding the history and the recent popularity of this problem, no estimate of the constants involved is available in the literature (except for the case $d=1$). Specially significant is the lack of this information when $d\geq 3$ because it is not clear, in particular, what is the behavior of $\Rave$ as a function of $d$ and for $d\to +\infty$.

In this paper, we concentrate on regular grids constructed on $d$-dimensional tori as a benchmark example. Their interest is motivated by the ability to intuitively capture the notion of dimension and by their nice mathematical properties: recent applications in network systems include~\cite{SK-EGF:10,RC-FF-AS-SZ:08, RC-FG-SZ:09, FG-SZ:11,BB-MJ-PM-SP:12}.
On such toroidal grids, we sharpen the above statements. 
Firstly, in dimension $d=2$ we compute the asymptotic proportionality constant and we provide tight estimates that allow us to study the asymptotic behavior when the grid sides are unequal. Secondly, in toroidal grids with $d \geq 3$ we show that, when the side lengths tend to infinity, the average effective resistance is of order $1/d$. In fact, we conjecture that the order $1/d$ is valid for finite side lengths too. 

Our analysis hinges on two facts: firstly, the average effective resistance can be computed using the eigenvalues of the Laplacian matrix associated to the graph; secondly, an explicit formula is available for the Laplacian eigenvalues of toroidal $d$-grids. Similar approaches have been taken elsewhere in the literature, namely in~\cite{WU:04} and in~\cite{BB-MJ-PM-SP:12}. The paper~\cite{WU:04} computes the effective resistances between pairs of nodes in $d$-dimensional grids by explicit formulas. Our work, instead, concerns estimates of average effective resistances in toroidal grids and their asymptotics for large $N$. The paper~\cite{BB-MJ-PM-SP:12} also estimates the average resistance for large $N$: in comparison, the novelty of our work resides in more accurate estimates of the quantities involved, which are essential to capture the features of high-dimensional and irregular grids.

\subsection*{Paper structure}
The rest of this paper is organized as follows. In Section~\ref{sec:prbm} we formally state our problem and we present and discuss our main results. Their detailed derivation is provided in Section~\ref{sec:eig-proofs}, which also contains a mean-field approximation of the average resistance in dimension $d$.	Finally, in Section~\ref{sec:concl} we draw some conclusions about our work and future research.

\section{Problem statement and main results}\label{sec:prbm}

We consider an undirected graph $G=(V,E)$, where $V$ is a finite set of vertices and $E$ is a subset 	of unordered pairs of distinct elements of $V$ called edges. We assume the graph to be connected and we think of it as an electrical network with all edges having unit resistance. Given two distinct vertices $u,v\in V$, the effective resistance between $u$ and $v$ is defined as follows. Let there be a unit input current at node $u$ and a unit output current at node $v$: using Ohm's and Kirchoff's law, a potential $W$ is then uniquely defined at every node (up to translation constants). We then define the effective resistance as $\Reff(u,v):=W_u-W_v$. 
Consequently, the average effective resistance of $G$ is defined as 
	\begin{align} \label{eq:rave-def}
		\Raveb{G}:=\frac{1}{2N^2}\sum\limits_{u,v\in V}\Reff(u,v),
	\end{align}
	where $N=|V|$ is the size of the graph. 

\subsection{Toroidal $d$-dimensional grids}	We now formally define the class of graphs we deal with.
	Consider the cyclic group $\Z_M$ of integers modulo $M$ and the product group $\Z_{M_1}\times\cdots\times\Z_{M_d}$. 
	Let $e_j\in\Z_{M_1}\times\cdots\times\Z_{M_d}$ be the vector with all $0$'s except $1$ in position $j$ and define
	$S=\{\pm e_j\,|\, j=1,\dots ,d\}$. 
	We define as the toroidal $d$-grid over $\Z_{M_1}\times\cdots\times\Z_{M_d}$ 
	the graph $T_{M_1,\dots ,M_d}=(\Z_{M_1}\times\cdots\times\Z_{M_d}, E_{M_1,\dots ,M_d})$ where
	$$E_{M_1,\dots ,M_d}:=\l\{\{(x_1,\dots ,x_d),(y_1,\dots ,y_d)\}\;\l\vert \; (x_1-y_1,\dots ,x_d-y_d)\in S\r.\r\}$$
	%
	In other words, we call toroidal $d$-grids those graphs where the vertexes are arranged on a Cartesian lattice in $d$ dimensions, which has sides of length $M_1, \ldots, M_d$ and has edges between any vertex and its $2d$ nearest neighbors, with periodic boundary conditions. The total size of the graph is $N = M_1\times\cdots\times M_d$. 
	%
In the special case $M_1=\dots =M_d$, i.e., when all the $M_i$ are equal to a specific $M$, we will use the notation $\TMd$ instead of $T_{M,\dots ,M}$.
%
In the special case when $M=2$, we actually obtain degenerate grids on $\Z_2^d$, which are called hypercubes of dimension $d$ and denoted by $\Hyp{d}$:
	%
	note that the size of $\Hyp{d}$ is $N=2^d$ and the degree of each vertex is $d$.
	%


\subsection{Asymptotic results}

	We start by recalling the simple case $d=1$, where the effective resistance can be directly computed. From the standard properties of series and parallel connections of resistors~\cite[pages 119--120]{LPW:08}, one can see that
$\Reff(v_0,v_0+l)=\frac{l(M-l)}{M}$ and thus
	\begin{align}
	    \Raveb{\TM} &= \frac{1}{2M} \sum_{l=1}^{M-1} \frac{l(M-l)}{M} 
		= \frac{M}{12} - \frac{1}{12M}. \label{eq:Reff1exact}
	\end{align}
	This formula leads to the asymptotic relation
	\begin{align*}
		\Raveb{\TM} \sim\frac{M}{12}\,\quad {\rm for}\; M\to +\infty.
	\end{align*}
	When $d \ge 2$, we prove in this paper that the following asymptotic relations hold.
	\begin{theorem}[Asymptotics]\label{thm:asymp}
		Let $\TMd$ be 
		the toroidal grid in $d \geq 2$ dimensions, with each side length being equal to $M$, 
		and let  $\Rave(\TMd)$ be its average effective resistance.
		Then,
		\begin{equation}\label{eq:Reff2}
			\Raveb{\TMdue}\sim\frac{1}{ 2 \pi}\ln M\quad{\rm for}\; M\to +\infty
		\end{equation}
		and
		\begin{equation}\label{eq:Reffd}
			\lim\limits_{M\to +\infty}\Raveb{\TMd}=\Theta\left( \frac{1}{d}\right)\quad{\rm for}\; d\to +\infty.
		\end{equation}
	\end{theorem}%
	The relations (\ref{eq:Reff2}) and (\ref{eq:Reffd}) follow immediately from the estimates provided below in Theorems~\ref{thm:2-torus} 
	and~\ref{thm:d-torus}.
	Furthermore, we conjecture that the statement (\ref{eq:Reffd}) can be sharpened as follows.
	\begin{conj}
			$$\Raveb{\TMd}=\Theta\left(\frac{1}{d}\right)\quad{\rm for}\; d\to +\infty\,,\; M\;{\rm fixed.}$$
	\end{conj}%
	At the moment we can only prove such a result in the degenerate case $M=2$, 
	corresponding to a hypercube, where 
\begin{align}	\label{ReffHypAsy}
	\Raveb{\Hyp{d}} \sim \frac1d\qquad \textup{ as } d\to\infty.	
\end{align}

\begin{figure*}
	\centering
	\fbox{%
		\psfrag{raAN}{\scriptsize $\TM$}
		\psfrag{ra2T}{\scriptsize $\TMdue$}
		\psfrag{ra3T}{\scriptsize $\TMtre$}
		\psfrag{ra4T}{\scriptsize $\TMquattro$}
		\psfrag{size}{\scriptsize Size \!$N$}
		\psfrag{rave}{\scriptsize $\Rave$}
		\includegraphics[width={\figwidthduo}, keepaspectratio=true]{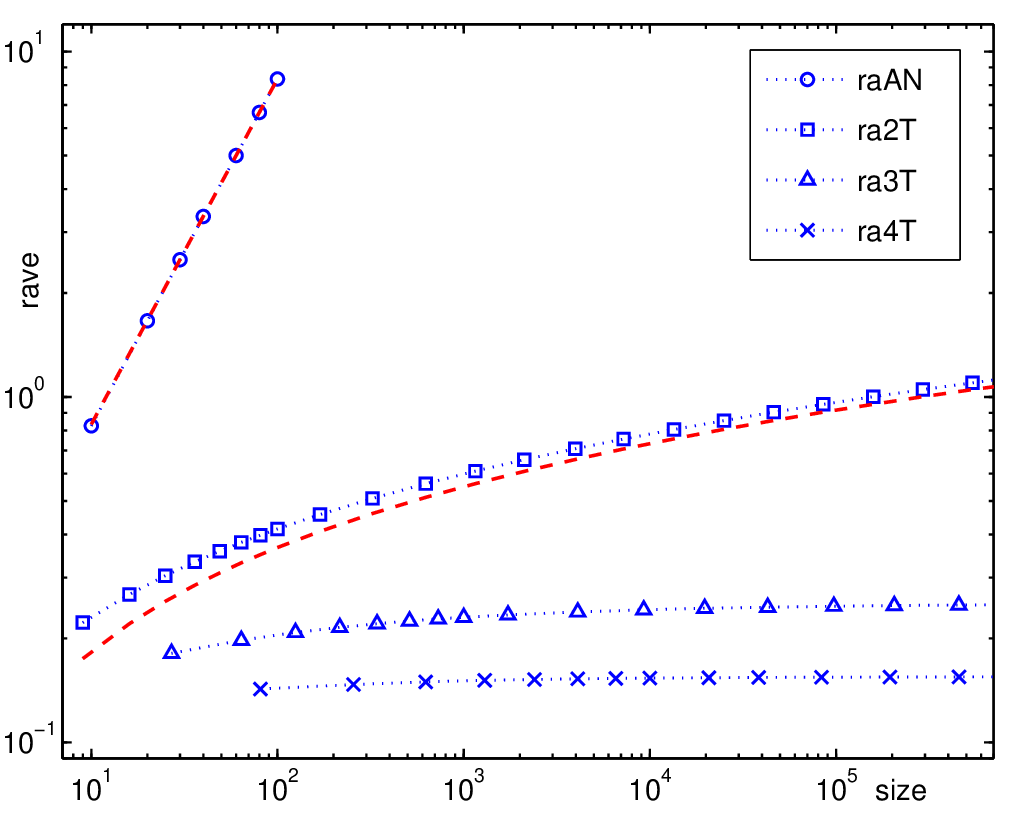}%
	}\hspace*{2pt}
	\fbox{%
		\psfrag{Rave}{\scriptsize$\Rave$}
		\psfrag{d}{\!\scriptsize Dimension $d$}
		\psfrag{M=3}{\tiny $M\!=\!\phantom{0}3$}
		\psfrag{M=4}{\tiny $M\!=\!\phantom{0}4$}
		\psfrag{M=5}{\tiny $M\!=\!\phantom{0}5$}
		\psfrag{M=8}{\tiny $M\!=\!\phantom{0}8$}
		\psfrag{M=16}{\tiny $M\!=\!16$}
		\psfrag{M=32}{\tiny $M\!=\!32$}
		\includegraphics[width={\figwidthduo},	keepaspectratio=true]{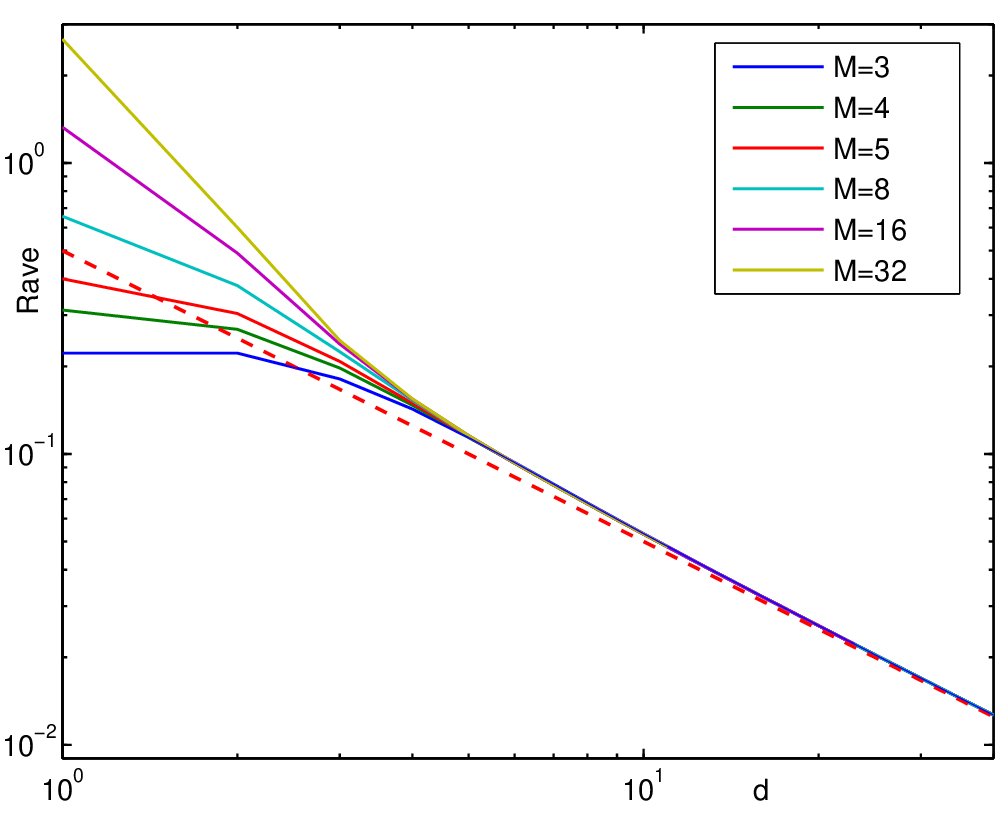}%
	%
	}%
	\caption{Left: $\Rave$ in low dimensional toroidal grids, as function of the size $N = M^d$, 
	with the dashed lines representing the asymptotic trends $N / 12$ and $\frac{1}{4\pi}\log N$. 
	Right: $\Rave$ in high dimensional toroidal grids, as function of the dimension $d$, 
	with the dashed line representing the trend $\frac{1}{2d}$. \label{fig:Rave-plots} }  
\end{figure*}

	Our results and conjecture are corroborated by numerical experiments, which are summarized in Figure~\ref{fig:Rave-plots}.
	The left plot of Figure~\ref{fig:Rave-plots} shows the average effective resistances $\Rave$ of four families 
	of low-dimensional graphs as functions of the total size $N$ of the graphs:
	$\Rave(\TM)$ and $\Rave(\TMdue)$ follow the predicted linear (\ref{eq:Reff1exact})
	and logarithmic (\ref{eq:Reff2}) asymptotic trends, whereas $\Rave( T_{M^3}) $ and $\Rave(T_{M^4})$ tend to a finite limit. 
	The right plot of Figure~\ref{fig:Rave-plots} instead regards high-dimensional graphs and shows that 
	$\Rave$ decreases with $d$, when the side lengths $M$ are kept fixed. 
	If $d \geq 5$, then $\Rave(\TMd)$ for different $M$ are roughly equal and inversely proportional to $2d$.
	This plot supports our conjecture that $\Rave(\TMd)$ is of order $1/d$, independent of $M$.

\subsection{Estimates for finite toroidal grids}
This subsection contains tight estimates of the average resistance in dimension $d$. These novel results are key to obtain the asymptotic relations presented above.
We begin with a pair of estimates in dimension two.
\begin{theorem}[Torus $\TML$]\label{thm:2-torus}
 Let $\TML$ be the toroidal grid in two dimensions with side lengths $\Muno$ and $\Mdue$, 
		and let $\Rave(\TML)$ be its average effective resistance.
		Suppose $4\le \Muno\le \Mdue$. Then,
		\begin{align*}
			\Raveb{\TML}	&	\leq	\frac{1}{2\pi} \log \Mdue +\frac{1}{12}\frac{\Mdue}{\Muno}+ 1	\\
			\Raveb{\TML}	&	\geq	\max \l\{ \, \frac{1}{12} \frac{\Mdue}{\Muno} - \frac{1}{24}  \enspace ; \enspace
								\frac{1}{2\pi}\log \Muno  - \frac1{12}\frac{\Mdue}{\Muno}- \frac{1}{2}  \,\r\}.
		\end{align*}				
	\end{theorem}
	
	%
In order to understand the consequences of Theorem~\ref{thm:2-torus}, it is useful to fix specific relations between $\Mdue$ and $\Muno$ and study 
	the asymptotic behavior when the size $N = \Muno \times \Mdue$ of the graph tends to infinity.
	%
	Preliminarily, we observe that in the lower bound of Theorem~\ref{thm:2-torus}, 
	the former expression dominates when $\Muno$ and $\Mdue$ grow with different rates, 
	while the latter dominates when $\Muno$ and $\Mdue$ have the same rate of growth.
	We then consider the following three relations between $\Muno$ and~$\Mdue$. 
	\begin{enumerate}
		\item 	$\Muno=\kcost$, $\Mdue=N/\kcost$. Then,
			$$\frac{1}{12} \frac{N}{\kcost^2} - \frac{1}{24} \leq 
			   \Raveb{ T_{\kcost , N/\kcost}  } 
			   \leq \frac{1}{12} \frac{N}{\kcost^2} + \frac{1}{2\pi}\log N + 1 .$$
			In this case, $ \Raveb{ T_{\kcost , N/\kcost} } \sim\frac{N}{12c^2}$ as $N \to +\infty$: 
			we may interpret this linear growth as reminiscent of the one-dimensional case.	
		\item	$\Muno=\sqrt[\kcost ]{N}$, $\Mdue=\sqrt[\kcost ]{N^{\kcost -1}}$ with $\kcost >2$. Then,
			$$ \frac{1}{12} N^{\frac{\kcost -2}{\kcost } } - \frac{1}{24} \leq
				\Raveb{ T_{ \sqrt[\kcost ]{N} , \sqrt[\kcost ]{N^{\kcost -1}} } } 
				\leq \frac{1}{12} N^{\frac{\kcost -2}{\kcost }} + \frac{1}{2\pi} \frac{\kcost -1}{\kcost }\log N + 1. $$
			In this case
			$\Raveb{ T_{ \sqrt[\kcost ]{N} , \sqrt[\kcost ]{N^{\kcost -1}} } }\sim N^{\frac{c-2}c}/12$ as $N \to +\infty$,
			which is sub-linear and proportional to the ratio between $\Mdue$ and $\Muno$.
		\item	$\Muno=\sqrt{N/\kcost}$, $\Mdue=\sqrt{\kcost N}$ with $\kcost  = \frac{\Mdue}{\Muno}$. Then,
			$$\frac{1}{4\pi} \log N - \frac{\log \kcost }{4\pi}  - \frac{\kcost }{12} - \frac{1}{2} \leq 
			  \Raveb{ T_{ \sqrt{N/\kcost}  , \sqrt{\kcost N}  } }
			  \leq \frac{1}{4\pi}\log N + \frac{\kcost}{12}   +  \frac{\log \kcost}{4\pi} + 1 .$$
			In this case,  	
			$  \Raveb{ T_{ \sqrt{N/\kcost}  , \sqrt{\kcost N}  } } \thicksim  \frac{1}{4\pi}\log N $ as $N \to +\infty$. 
			That is, taking $\Muno$ proportional to $\Mdue$ makes 
			$\Raveb{\TML}$ grow logarithmically with $N$: this order of growth must be contrasted against the linear growth that characterizes one-dimensional graphs and against the two previous examples. In fact, this is the lowest asymptotic average effective resistance reachable by a bidimensional toroidal grid.  

	\end{enumerate}


		\bigskip
Next, we provide a pair of bounds valid when $d\ge3$:  for simplicity, we assume that the lengths along each of the $d$ dimensions are all equal to $M$.
	
	\begin{theorem}[Torus $\TMd$]\label{thm:d-torus}
 Let $\TMd$ be the toroidal grid in $d\ge 3$ dimensions, with each side length being equal to $M$, and let $\Rave(\TMd)$ be its average effective resistance. Provided $M\geq 4$, it holds that:
		\begin{align*}
			\Raveb{\TMd}	&\leq	\frac{8}{d+1} \l( 1 + \frac{1}{M} \r)^{d+1} 
										+ \frac{d}{4 M^{d-2}} \l( \frac{1}{3} + \frac{(d-1) \log M}{\pi} \r) \\
			\Raveb{\TMd}	&\geq 	\frac{1}{4 d}.
		\end{align*}
		%
	\end{theorem}
	Notice that if $d\geq 3$ is fixed and $M$ diverges, then Theorem~\ref{thm:d-torus} yields
	$\Raveb{\TMd} =\Theta(1)$ as $M\to +\infty$.
	This fact is well-known: the difficulty here lies in finding a tight upper bound, which can reveal the dependence on $d$ and imply~\eqref{eq:Reffd}. 
	
%
%
%

		\medskip
We conclude the presentation of our main results with the relevant estimates for the hypercube, corresponding to the case $M=2$.
	\begin{theorem}[Hypercube]\label{thm:hypercube}
		Let $\Hyp{d}$ be a $d$-dimensional hypercube graph, and $\Rave(\Hyp{d})$ be its average effective resistance.
		When $d\ge2$, the following estimates hold:
		\begin{align*} 
			\frac12\frac{1}{d+1} \le \Raveb{\Hyp{d}} \leq \frac{2}{d+1}.
		\end{align*}
	\end{theorem}

\section{Resistance and eigenvalues}\label{sec:eig-proofs}

We have seen in the previous section that the average effective resistance of the one-dimensional ring graph can be computed from the effective resistance between any pair of nodes. 
Indeed, in that case, effective resistances can be directly computed using simple properties of electrical networks. However, this approach is not viable for $d$-dimensional tori with $d\geq 2$. Instead, we can rely on the fact that for any graph $\Rave(G)$ can be expressed in terms of its Laplacian eigenvalues.  
	Given a graph $G$, the Laplacian of $G$, $L(G)\in\R^{V\times V}$ is the matrix defined by 
	$$L(G)_{uu}=|\{v\in V\,|\, \{u,v\}\in E\}|\,,\;\;  
	L(G)_{uv }=\left\{\begin{array}{ll} -1 \; 
	&{\rm if}\, \{u,v\}\in E\\ 0 \; &{\rm otherwise}\end{array}\right. u\neq v$$
	It is well known that its eigenvalues can be ordered to satisfy 
	$0=\lambda_1<\lambda_2\leq\cdots \leq \lambda_N$ and the following relation holds true~\cite[Eq.~(15)]{AG-SB-AS:08}
	\begin{align}\label{eq:rave-eig}
		\Raveb{G}=\frac{1}{N}\sum\limits_{i\geq 2}\frac{1}{\lambda_i}
	\end{align}
	 We are going to use (\ref{eq:rave-eig}) in order to prove our results\footnote{	Note that using the Laplacian eigenvalues and eigenvectors it is possible to compute the effective resistance 
	between any pair of nodes~\cite[Eq.~(11)]{WU:04}:
	\begin{align*}
		\Reff(v,u) = \sum\limits_{i\geq 2} \frac{1}{\lambda_i} \l| \psi_i(v) - \psi_i(u) \r|^2
	\end{align*}
	where $\psi_i(v)$ is the component $v$ of the eigenvector associated to the eigenvalue $\lambda_i$ of the Laplacian of $G$.
	Actually, from this formula and the definition of $\Rave(G)$ one easily deduces (\ref{eq:rave-eig}), which only requires the knowledge of the eigenvalues. }.
Indeed, the eigenvalues of the Laplacian can be exactly computed for the toroidal grid $T_{M_1,\dots ,M_d}$ using a discrete Fourier transform~\cite{FG-SZ:11}
	\begin{equation}\label{eq:eigenvalues}
		\lambda_{\bf h}=\lambda_{h_1,\dots ,h_d}=2d -2\sum\limits_{i=1}^d\cos\frac{2\pi h_i}{M_i}\,,\quad 
		{\bf h}= (h_1,\dots h_d)\in \Z_{M_1}\times\cdots\times\Z_{M_d}.
	\end{equation}
This formula leads to the following key expression
	\begin{equation}\label{eq:rave-green}
		\Raveb{ T_{M_1,\dots ,M_d} }=\frac{1}{M_1\cdots M_d}\sum\limits_{\hh\neq \0}
			\frac{1}{ 2d-2\sum\limits_{i=1}^d\cos \l( \frac{2\pi h_i}{M_i}\r) }
	\end{equation}
on which most of our derivations are based (excluding Section~\ref{sect:hypercube}). 

\subsection{Bounds for the $2$-torus $\TML$}

	We provide here the proof of the Theorem~\ref{thm:2-torus}. 
	As explained before, we resort to the Laplacian eigenvalues, which for $\TML$ read 
	$\lambda_{i,j} =  4-2\cos(2\pi i  / \Muno)-2\cos(2\pi j  / \Mdue)$ 
	with $i \in \{0,\ldots,\Muno-1\}$ and $j \in \{0,\ldots,\Mdue-1\}$.
	Hence, 
	\begin{align*}
		\Raveb{\TML}     
			= \frac{1}{\Muno\Mdue} \sum_{(i,j)\neq \0} \frac{1}{4-2\cos(2\pi i  / \Muno)-2\cos(2\pi j  / \Mdue)}.
	\end{align*}
	In order to estimate this quantity, we are going to interpret certain partial sums as upper/lower Riemann sums of suitable integrals, similarly to what is done in~\cite{BB-MJ-PM-SP:12}. However, it will be essential to single out some ``one-dimensional'' contributions to the overall sum. 
To this goal, we remind that
	$$ \Rave(\TM) = \frac{1}{M}\sum_{i\geq 1} \frac{1}{ 2 - 2\cos(2\pi i /M)},$$
	since the eigenvalues of $\TM$ are $\lambda_i = 2 - 2\cos(2\pi i /M)$ with $i \in \{0,\ldots,M-1\}$.

	\medskip
	\noindent
{\it Proof of Theorem~\ref{thm:2-torus}:} 
	 In order to prove the upper bound we rewrite $\Raveb{\TML}$ as
		\begin{equation}\label{eq:decompose1} 
			\Raveb{\TML} =\frac{1}{\Mdue} \Raveb{\TMuno} + \frac{1}{\Muno} \Raveb{\TLuno}+  \Raveoint(\TML) 
		\end{equation}
		where
		$$ \Raveoint(\TML) =\frac{1}{\Muno\Mdue} \sum_{i\neq0}^{}\sum_{j\neq0}^{} \frac{1}{\lambda_{i,j}}$$
		The first two terms in (\ref{eq:decompose1}) are easily bounded with the explicit formula (\ref{eq:Reff1exact}):
		\begin{align}\label{eq:exact1d}
			\frac{1}{\Mdue} \Raveb{\TMuno} + \frac{1}{\Muno} \Raveb{\TLuno} & \leq  \frac{\Muno}{12 \Mdue} + \frac{\Mdue}{12 \Muno}
		\end{align}
		Concerning $ \Raveoint(\TML)$, by symmetry it holds that:
		\begin{align*} 
			\Raveoint(\TML) & = \frac{1}{\Muno\Mdue} \sum_{i=1}^{\Muno-1}\sum_{j=1}^{\Mdue-1} \frac{1}{\lambda_{i,j}} 
			\leq \frac{4}{\Muno\Mdue} \sum_{i=1}^{\lfloor \Muno/2 \rfloor}\sum_{j=1}^{\lfloor \Mdue/2 \rfloor} \frac{1}{\lambda_{i,j}}.
		\end{align*}
		Consider the function 
		\begin{align}\label{eq:define-f}
			f(x,y) = \frac{1}{4-2\cos(2\pi x) - 2\cos(2\pi y)}  
		\end{align} 
		and notice that $\frac{1}{\lambda_{i,j}} = f\left(\frac{i}{\Muno},\frac{j}{\Mdue}\right)$.
		 For a fixed $\bar{y}$, $f$ is decreasing for $x\in (0,1/2]$, and viceversa for fixed $\bar{x}$,  
		$f$ is decreasing for $y\in (0,1/2]$. It follows that, for each pair $i,j$ 
		with $1 \leq i \leq \lfloor \Muno/2 \rfloor $ and $1 \leq j \leq \lfloor \Mdue/2 \rfloor $,
		\begin{align*}
			\frac{1}{\Muno\Mdue} \frac{1}{\lambda_{i,j} } \leq \int_{\frac{j-1}{\Mdue}}^{\frac{j}{\Mdue}} 
			\int_{\frac{i-1}{\Muno}}^{\frac{i}{\Muno}} f(x,y) \ud x \ud y.
		\end{align*}
		\begin{figure}
			\centering
			\includegraphics[width= 140pt ]{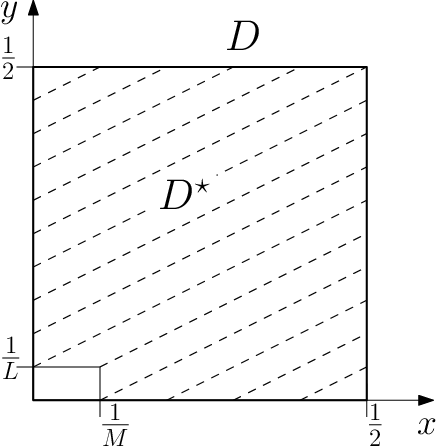}
			\includegraphics[width= 140pt ]{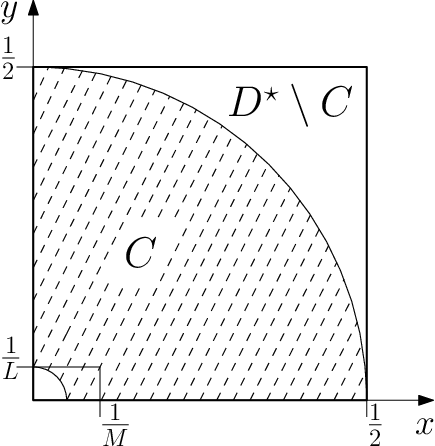}
			\caption{The regions $D$, $D^\star$, and $C$, which are useful in the proof of 
						 the upper bound of Theorem~\ref{thm:2-torus}. }
			\label{fig:region1}
		\end{figure}
		 Define the region $D =[0, 1/2]\times[0, 1/2]$ and  $D^\star =D \setminus \l([0, 1/\Muno]\times[0, 1/\Mdue]\r)$ 
		as in Figure~\ref{fig:region1}~(left),  to estimate 
		\begin{align}
			\Raveoint(\TML)&=\frac{4}{\Muno\Mdue} \sum_{i=1}^{\lfloor \Muno/2 \rfloor}
										\sum_{j=1}^{\lfloor \Mdue/2 \rfloor} \frac{1}{\lambda_{i,j}} \nonumber \\
				& \leq \frac{4}{\Muno\Mdue} f\left(\frac{1}{\Muno},\frac{1}{\Mdue}\right) 
							+ 4\iint_{D^\star} f(x,y) \ud x \ud y\label{eq:torus-riemann-sum}.
		\end{align}
		The term for $i=1,j=1$ is kept aside, because of the singularity in the origin. 
		%
		Next, instead of computing the integral in~\eqref{eq:torus-riemann-sum} in closed form, we observe that 
		\begin{align*}
				f(x,y) &= \frac{1}{4-2\cos(2\pi x)-2\cos(2\pi y)} \\
				&\leq \frac{1}{(2\pi x)^2 + (2\pi y)^2 - \frac{(2\pi x)^4}{12} - \frac{(2\pi y)^4}{12}} \\
				&\leq \frac{1}{(2\pi)^2 (x^2+y^2) - \frac{(2\pi)^4}{12} (x^2 + y^2)^2 }  = g(\sqrt{x^2 + y^2}),
		\end{align*}
		where we defined the function $\map{g}{(0,\frac{\sqrt{3}}{\pi})}{\R^+}$ as 
		\begin{equation}\label{eq:g-def}
		g(r)=\frac{1}{4\pi^2 r^2 \left( 1 - \frac{\pi^2}{3}r^2 \right)}.
		\end{equation}
		Unfortunately, $g$ does not provide an useful upper bound because it 
		has a singularity in $\frac{\sqrt{3}}{\pi}$.
		We instead use the following continuous modification
		\begin{align*}
		\tilde{g}{(\rho)} = \left\{ \begin{array}{ll}
			\frac{1}{4\pi^2 \rho^2 \left( 1 - \frac{\pi^2}{3}\rho^2 \right)}  & \text{if}\quad 0 < \rho < \frac{1}{2}\\
			\frac{1}{\pi^2 \left( 1 - \frac{\pi^2}{12}\right)} & \text{if}\quad\rho \geq \frac{1}{2},\\
		\end{array} \right.
		\end{align*}
		which is decreasing in $\l(0 , \frac{\sqrt{3}}{\sqrt{2} \pi}\r)$ and such that 
		$f(x,y) \leq \tilde{g}{\l( \sqrt{x^2 + y^2} \r)}$ for all $(x,y) \in D .$
		We now use this bound to estimate the right-hand side of (\ref{eq:torus-riemann-sum}). 
		Regarding the first term, using that $\Mdue\geq \Muno\geq 4$, we obtain
		\begin{equation}\label{eq:stima-g-tilde}
			\frac{4}{\Muno\Mdue}\,\tilde g\!\left(\sqrt{\frac{1}{\Mdue^2} + \frac{1}{\Muno^2}} \right) 
			\leq \frac{4}{\Muno\Mdue}\tilde g(1/\Muno)
			\leq \frac{2}{\pi^2}\frac{\Muno}{\Mdue}\,.
		\end{equation}
		On the other hand, defining $C=\setdef{(x,y)\in \R^2}{\frac1{\Mdue^2}\le x^2+y^2\le \frac14}$ 
		as illustrated in Figure~\ref{fig:region1}~(right), 
		we can estimate the second term with polar coordinates:	
		\begin{align}
			4 \iint_{D^\star} f(x,y) \ud x \ud y  
				&	= 	4 \iint_{D^\star} \tilde{g}(\rho) \rho \ud \rho \ud \theta  \nonumber \\
				&\leq	\displaystyle 4 \iint_C \tilde{g}\l(\rho\r)  \rho \ud \rho \ud \theta + 4\iint_{D^\star \setminus C}  
							\tilde{g}\l(\rho\r)  \rho \ud \rho \ud \theta  \nonumber \\
				&\leq	\displaystyle 4 \int_0^{\frac{\pi}{2}} \int_{\frac{1}{\Mdue}}^{ \frac{1}{2} } 
							\frac{1}{4\pi^2 \rho^2 \left( 1 - \frac{\pi^2}{3}\rho^2 \right)}  \rho \ud \rho \ud \theta 
							+ \left(1-\frac{\pi}{4} \right) \tilde{g}\left(\frac{1}{2}\right) \nonumber \\
				& \leq	\displaystyle\frac{2}{\pi^2} \frac{\Muno}{\Mdue}  
							+\frac{1}{2\pi} \int_{\frac{1}{\Mdue}}^{1/2} \frac{1}{\rho - \frac{\pi^2}{3}\rho^3} \ud \rho  
							+ \frac{1}{6} \nonumber \\
				&	=	\displaystyle\frac{1}{2\pi}\left[\log \rho
							- \frac{1}{2}\log\left(1-\frac{\pi}{3}\rho^2\right)\right]_{\frac{1}{\Mdue}}^{1/2}  + \frac{1}{6} \nonumber \\
				&\leq \displaystyle \frac{1}{2\pi} \log \Mdue - \frac{1}{4\pi}\log\left(1-\frac{\pi}{12}\right)   + \frac{1}{6} \nonumber\\
				&\leq	\displaystyle \label{eq:stima-integrale} \frac{1}{2\pi} \log \Mdue + \frac{1}{5}.
		\end{align}
		Using bounds (\ref{eq:stima-g-tilde}) and (\ref{eq:stima-integrale}) in (\ref{eq:torus-riemann-sum}) we obtain
		\begin{equation}\label{eq:Rdotestim2}
			\Raveoint(\TML) \leq \frac{1}{2\pi} \log \Mdue + \frac{2}{\pi^2}\frac{\Muno}{\Mdue}+ \frac{1}{5}
		\end{equation}
		Using now (\ref{eq:Rdotestim2}) and (\ref{eq:exact1d}) in (\ref{eq:decompose1}), we finally get
		\begin{align*}
			\Raveb{\TML} & \leq \frac{1}{2\pi} \log \Mdue + \frac{\Mdue}{12\Muno} 
				+ \left( \frac{2}{\pi^2}+\frac{1}{12}\right) \frac{\Muno}{\Mdue} + \frac{1}{5}
		\end{align*}
		and the thesis follows since $\frac{\Muno}{\Mdue}\leq 1$.
		
		The first estimate of the lower bound can be proved easily: it is enough to neglect in the expression of 
		$\Raveb{\TML}$ all terms that have $i > 0$ or $j > 0$. Then,
		\begin{align*}
			\Raveb{\TML}& \geq \frac{1}{\Mdue} \Raveb{\TMuno}	+ \frac{1}{\Muno} \Raveb{\TLuno} \\
							&= \frac{1}{\Mdue}\left(\frac{\Muno}{12} - \frac{1}{12\Muno}\right) 
							+ \frac{1}{\Muno}\left(\frac{\Mdue}{12} - \frac{1}{12\Mdue}\right) \\
						& \geq \frac{1}{12}\left( \frac{\Mdue}{\Muno} + \frac{\Muno}{\Mdue}\right) - \frac{1}{6\Mdue \Muno} 
						  \geq \frac{1}{12}\frac{\Mdue}{\Muno} - \frac{1}{24} 
		\end{align*}
		
		To prove the second estimate, we use an approach similar to that
		of the upper bound.
		 Since a symmetric domain is convenient, we define the index sets
		\begin{align*}
			\Gamma_+ &= \Z_\Muno \times \Z_\Mdue \quad \backslash \quad \{(0,0)\} \\ 
			\Gamma^+ &= \Gamma_+ \quad \cup \quad \{\Muno\} \times \{1,2,\ldots, \Mdue-1\}\quad 
			                           \cup \quad \{1,2,\ldots, \Muno-1\} \times \{\Mdue\}\quad
		\end{align*}
		 to write
		\begin{align}\label{eq:decompose2}
			\Raveb{\TML} & = \frac{1}{\Muno\Mdue} \sum_{\Gamma_+} \frac{1}{\lambda_{i,j}} 
			                 = \Raveoext(\TML)
			           -\frac{1}{\Mdue} \Raveb{\TMuno} -\frac{1}{\Muno} \Raveb{\TLuno}
		\end{align}
		where $\Raveoext(\TML)= \frac{1}{\Muno\Mdue} \sum_{\Gamma^+} \frac{1}{\lambda_{i,j}} $. 
		\begin{figure}
			\centering
			\includegraphics[width= 170pt ]{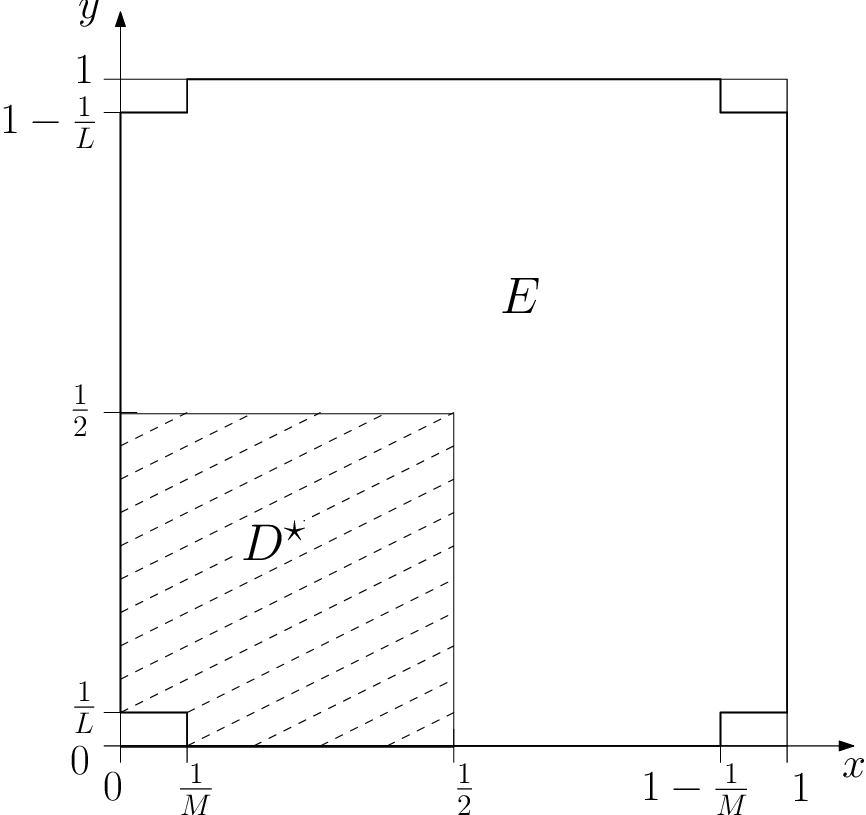}
			\includegraphics[width= 170pt ]{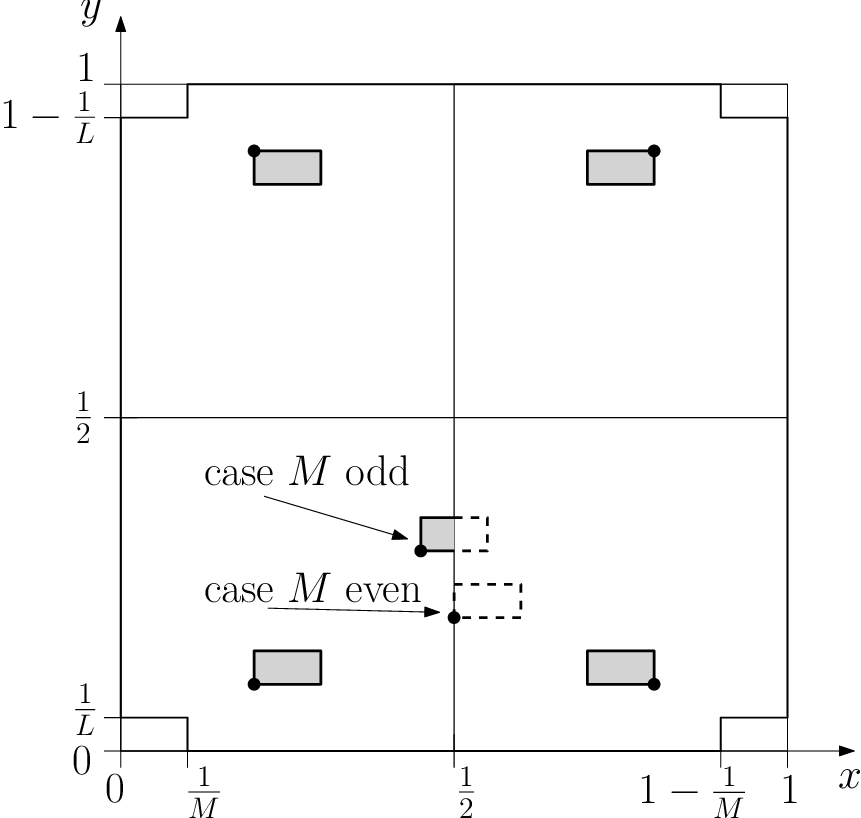}
			\caption{Left plot: Regions $E$ and $D^\star$. Right plot: 
			In order to illustrate how the Riemann sum is built, dots on the corners of the grey rectangles 
			indicate the interpolation points, whose values are assumed on each rectangle. The contributions 
			of the dashed parts of the rectangles are disregarded in the integral, without compromising the 
			validity of inequality~\eqref{eq:riemann-lower-torus}.\label{fig:region3}} 
		\end{figure}
		To estimate   $\Raveoext(\TML)$, we consider the function $f(x,y)$ as defined in the proof of 
	 the upper bound 
		and the domain $E$, defined 
		(Figure~\ref{fig:region3}) as:
		$$ E = [0,1]\times [0,1] \setminus \left(\l( \l[ 0 , \frac{1}{\Muno} \r] \cup \l[ 1 - \frac{1}{\Muno} , 1 \r] \r) \times 
		\l( \l[ 0 , \frac{1}{\Mdue} \r] \cup  \l[ 1 - \frac{1}{\Mdue} , 1 \r] \r)\right)\,,$$
		and we notice that
		\begin{align} \label{eq:riemann-lower-torus}
			\Raveoext(\TML) &\geq \iint_E f(x,y) \ud x \ud y
			= 4 \iint_{D^\star} f(x,y) \ud x \ud y
		\end{align} 
		where the equality exploits the symmetry of $f$.
		Since $f(x,y)\geq (4\pi^2)^{-1}(x^2+y^2)^{-1}$, we obtain
		\begin{align*}
			\Raveoext(\TML) & \geq \frac{1}{\pi^2} \iint_{D^\star} \frac{1}{x^2 + y^2} \ud x \ud y \\
			&\geq \frac{1}{2\pi} \int_\delta^{1/2} \frac{1}{\rho^2} \rho \ud \rho 
			= \frac{1}{2\pi}\left( \log( \delta^{-1}) -\log 2 \right)		,
		\end{align*}
		with $\delta = \sqrt{\frac{1}{\Muno^2} + \frac{1}{\Mdue^2}}$. 
		If we observe that $\frac{1}{\Muno^2} + \frac{1}{\Mdue^2} \leq \frac{2}{\Muno^2}$, we get
		\begin{align}\label{eq:raveestim}
			\Raveoext(\TML) &\geq \frac{1}{2\pi}\log(\Muno) - \frac{1}{4}, 
		\end{align} 
		Using now (\ref{eq:raveestim}) inside (\ref{eq:decompose2}) 
		together with the exact calculation (\ref{eq:Reff1exact}), we finally obtain
		\begin{align*}
			\Raveb{\TML} &\geq \frac{1}{2\pi}\log(\Muno)  - \frac{\Mdue}{12\Muno} - \frac{\Muno}{12\Mdue} - \frac{1}{4} 
							\geq \frac{1}{2\pi}\log(\Muno)  - \frac{\Mdue}{12\Muno} - \frac{1}{2}. 
		\end{align*}
		This inequality concludes the proof
		 of the second estimate for the lower bound, and hence the proof of the theorem.
	\hfill$\square$

\subsection{Continuous approximation of $\Rave(\TMd)$}
	We consider here the quantity $\Ravehydro{d}$, defined as:
	\begin{align} \label{eq:ravehydro-d}
	     \Ravehydro{d} := \int_{ \l[0, 1 \r]^d} \frac{1 }{2 d - 2\sum_{i=1}^d \cos(2\pi x_i)} \ud \xx. 
	\end{align}
	and prove an upper and lower bound of order $1/d$.
	 In the proof of Theorem~\ref{thm:d-torus}, this quantity will play the role of a ``continuous'' approximation of $\Rave(\TMd)$.
	
	\begin{lemma}\label{thm:ravehydro-d} 
		If $d \geq 3$, then
		\begin{align*}
			\frac1{4d}\le \Ravehydro{d}\le\frac{4}{d}. 
		\end{align*}
	\end{lemma}
	\begin{proof}
		The lower bound is trivial: the integrand is not smaller than $\frac{1}{4 d}$ over all the domain.
		What follows is devoted to prove the upper bound.
		By symmetry
		\begin{align*}
		   	\Ravehydro{d} = 2^d \int_{ \l[0, \frac{1}{2} \r]^d}  \frac{1}{2 d - 2\sum_{i=1}^d \cos(2\pi x_i)} \ud \xx,
		\end{align*}
	 then we define the following three subsets of $\l[0, \frac{1}{2}\r]^d$, 
		\begin{align*}
			\AAA &= \l\{ \xx \in \l[ 0,\frac{1}{2}\r]^d \quad \textup{s.t.} \quad \Vert \xx \Vert_2 \leq \frac{1}{\pi} \r\}\\
			\BBB &= \l\{ \xx \in \l[ 0,\frac{1}{2}\r]^d \quad \textup{s.t.} \quad \Vert \xx \Vert_2 \geq \frac{1}{\pi} 
								\quad \textup{and}\quad x_i \leq \frac{1}{\pi} \quad \forall i\r\}\\
			\CCC &= \l\{ \xx \in \l[ 0,\frac{1}{2}\r]^d \quad \textup{s.t.}\quad \exists \quad x_i \geq \frac{1}{\pi} \r\}
		\end{align*}
		such that $ A \cup B \cup C = \l[0,\frac{1}{2}\r]^d $.
		Correspondingly, we define 
		\begin{align*}
			\Ok &= 2^d \int_{\AAA}  \frac{1}{2d - 2\sum_{i=1}^d \cos(2\pi x_i)} \ud \xx 
			\\
			\BZk &= 2^d \int_{\BBB} \frac{1}{2d - 2\sum_{i=1}^d \cos(2\pi x_i)} \ud \xx 		
			\\
			\Sk &= 2^d \int_{ \CCC } \frac{1}{2d - 2\sum_{i=1}^d \cos(2\pi x_i)} \ud \xx ,	
		\end{align*}
		so that
		$ \Ravehydro{d} = \Ok + \BZk + \Sk$
		
		We begin by a bound on $\Ok$. 
		First, we work on the denominator of the integrand, using the inequality 
		$1 - \cos x \geq \frac{x^2}{2} - \frac{x^4}{24}$ to show
		\begin{align*} 
			2 \sum_{i=1}^d \left(1  - \cos(2\pi x_i) \right) & \geq
			4 \pi^2 \sum_{i=1}^d  x_i^2 -\frac{16 \pi^4}{12} \sum_{i=1}^d x_i^4 \\
			& \geq 4 \pi^2 \left( \sum_{i=1}^d  x_i^2 -\frac{\pi^2}{3} \sum_{i=1}^d \sum_{j=1}^d x_i^2 x_j^2 \right) \\
			&= 4\pi^2 \left(1 - \frac{\pi^2}{3} \sum_{i=1}^d x_i^2 \right) \sum_{i=1}^d  x_i^2.
		\end{align*}
		 With the last expression, in polar coordinates we obtain
		\begin{align*}
			\Ok \leq & \, 2^d \int_{\AAA }\frac{1}{ 4\pi^2 \left(\sum_{i=1}^d  x_i^2 \right) 
												\left(1 - \frac{\pi^2}{3} \sum_{i=1}^d x_i^2 \right) }  \ud \xx \\
			= &  \int_0^{\frac{1}{\pi}}  \frac{ 2 \pi^{\frac{d}{2}} }{ \Gamma\left(\frac{d}{2}\right) } 
						\rho^{d-1}  \frac{1}{ 4\pi^2 \rho^2 \left(1 - \frac{\pi^2}{3} \rho^2 \right) } \ud \rho\\
			= &  \frac{ \pi^{\frac{d}{2} - 2} }{ 2 \Gamma\left(\frac{d}{2}\right) } 
									\int_0^{\frac{1}{\pi}} \frac{\rho^{d-3}}{1 - \frac{\pi^2}{3} \rho^2 } \ud \rho.
		\end{align*}
		The change of variables involving the Gamma function has cleared the singularity in zero, 
		and the new integrand is an increasing function. Then,
		\begin{align*}
			\Ok \leq & \frac{ \pi^{\frac{d}{2} - 2} }{2\Gamma\left(\frac{d}{2}\right) } 
			\int_0^{\frac{1}{\pi}}  \frac{\left(\frac{1}{\pi}\right)^{d-3}}{  
						\left[1 - \frac{\pi^2}{3} \left(\frac{1}{\pi} \right)  ^2 \right]  }\ud \rho
			=  \frac{3  }{4 \pi^{\frac{d}{2}} \Gamma\left(\frac{d}{2}\right) }.
		\end{align*}
		Since $x^{(1-\gamma)x-1}<\Gamma(x)$ if $x>1$ (see~\cite{ASS:10}),
		where $\gamma\simeq 0.577$ is the Euler-Mascheroni constant, 
		we have
		\begin{align}\label{eq:partial-regionA}
			\Ok \leq \frac{ 3 d }{8 \pi^\frac{d}{2} \left(\frac{d}{2} \right)^{(1-\gamma)\frac{d}{2}}  }.
		\end{align}
		  
		Next, we estimate $\BZk$. Recall definition~\eqref{eq:define-f} and  notice that the function
		$$f(\xx):=\frac{1}{2d - 2\sum_{i=1}^d \cos(2\pi x_i)}$$
		is decreasing in every direction $i$, when $\xx\in[0,\frac12].$
		Then, defining $g(\rho)$ as in~\eqref{eq:g-def}, we have
		\begin{align}\label{eq:partial-regionB}
			\BZk & \leq 2^d \mu(\BBB) g\l(\frac{1}{\pi}\r) \leq \frac{3}{8} \l(\frac{2}{\pi} \r)^d,
		\end{align}
		where $\mu(B)$ denotes the measure of $B$, and $\BBB\subset{\left[0,\frac1\pi\right]^d}.$
		  	
	 Finally, we consider $\Sk$. 
		Let $\Omega = \{0,1\}^d$ and for all $\omega \in \Omega$, define the set  $\CCC_\omega \subset \CCC$ as
		$\CCC_\omega = \{\xx \in \CCC \quad \textup{s.t.} \quad x_i \geq \frac{1}{\pi} \quad\textup{iff}\quad \omega_i=1 \}. $
		Clearly, $\bigcup_{\omega\neq \mathbf{0}} \CCC_\omega = \CCC.$ Then, 
		\begin{align*}
			\Sk 
			&= 2^d \sum_{\omega \neq\mathbf{0}} \int_{ \CCC_{\omega} }  \frac{1}{2d - 2\sum_{i=1}^d \cos(2\pi x_i)} \ud \xx.
		\end{align*}
		For a fixed $\omega\in \Omega$ we denote by $l_\omega$ the number of 1's in $\omega$ 
		(that is, the so-called Hamming weight of $\omega$), and notice that 
		$$\mu( \CCC_{{\omega}} ) = \left(\frac{1}{\pi} \right)^{d-l_\omega} \left( \frac{1}{2} -\frac{1}{\pi} \right)^{l_{\omega}}.$$ 
		Moreover, the function $f(\xx)$ is symmetric under permutations of the components of $\xx$. 
		Then, 
		$$f(\xx)\le f\left( \frac{1}{\pi} \omega\right)=\frac{1}{2(1 - \cos(2)) }\frac{1}{l_\omega}\qquad \text{if}\; \xx\in \CCC_\omega.$$
		Since clearly there are $\binom{d}{l}$ elements in $\Omega$ with Hamming weight $l$, we can argue that 
		\begin{align*}
			\Sk 
			& \le 2^d \sum_{l=1}^d
			\binom{d}{l} \frac{1}{2l(1-\cos(2))}\left(\frac{1}{\pi} \right)^{d-l}\left(\frac{1}{2} - \frac{1}{\pi}\right)^{l}\\
			& = \frac{1}{2(1 - \cos(2))}\sum_{l=1}^d \binom{d}{l} \left(\frac{2}{\pi} \right)^{d-l} 
				\left(1- \frac{2}{\pi} \right)^l \frac{1}{l} \\
			& \leq \frac{1}{(1 - \cos(2))(1-\frac{2}{\pi} )} \frac{1}{d+1}
		\end{align*}
		where the last inequality follows from standard manipulations on the binomials.
		This bound can be replaced by a simpler
		\begin{align}\label{eq:partial-regionC}
			\Sk \leq  \frac{3}{d}
		\end{align}
		and we are able to conclude the proof by combining~\eqref{eq:partial-regionA},
		\eqref{eq:partial-regionB}, and~\eqref{eq:partial-regionC} to get 
		$ \Ravehydro{d} = \Ok + \BZk + \Sk \leq \frac{4}{d} $.
	\end{proof}

\subsection{Bounds for the $d$-torus $\TMd$}

	We proceed with the proof of Theorem~\ref{thm:d-torus}, containing the bounds for $\Rave(\TMd)$ when $d\geq 3$. 
	Notice that, when all the side length are equal to $M$, the general expression (\ref{eq:rave-green}) becomes:
	\begin{align}
		\Raveb{\TMd} = \frac{1}{ M^d }\sum\limits_{\hh\neq \0}
			\frac{1}{ 2d-2\sum\limits_{i=1}^d\cos \l( \frac{2\pi h_i}{M}\r) }  \label{eq:rave-TMd}
	\end{align}

	\medskip
	\noindent
	{\it Proof of Theorem~\ref{thm:d-torus}:}
	 The lower bound can be easily proved by observing
		that $\forall\, \hh \neq \mathbf{0}$, $\frac{1}{\lambda_\hh} \geq \frac{1}{4 d}$.
		Moreover, since $\frac{1}{\lambda_{(1,0,\ldots,0)}} = \frac{1}{2 - 2 \cos(\frac{2\pi}{M})} \geq \frac{1}{2 d} $,
		$$\Raveb{\TMd} \geq \frac{1}{M^d} \l[ (M^d - 2) \frac{1}{4 d} + \frac{2}{4 d} \r] = \frac{1}{4 d}.$$

		In order to prove the upper bound,
		let us consider the terms in the sum (\ref{eq:rave-TMd}) for which $\hh \succ \mathbf{0}$, 
		i.e., those for which all $h_i > 0$. 
		Define
		$$ \Ravebo{\TMd} = \frac{1}{M^{ d}} \sum_{\hh \succ \mathbf{0}} 
		        \frac{1}{ 2 d - 2\sum_{i=1}^{d} \cos\l(\frac{2\pi h_i}{M}\r)} $$
		(where $\hh \succ \mathbf{0}$ means that $h_i > 0$ for all $i$),         
		and observe that
		$$ \Raveb{\TMd} = \sum_{m=1}^d \binom{d}{m} \frac{1}{M^{d-m}} \Ravebo{\TMm}.  $$  
		 It is crucial to observe that, 
		 with $\Ravehydro{m}$ defined at (\ref{eq:ravehydro-d}),
		$$ \Ravebo{\TMm} \leq  \Ravehydro{m}$$ 
		for any $m\ge1$, since we can see $\Ravebo{\TMm}$ as a lower Riemann sum of the integral. 
		 When $m \geq 3$, Lemma~\ref{thm:ravehydro-d} gives $$ \Ravebo{\TMm} \leq \frac{4}{m},$$
while for $m=2$ we use the bound (\ref{eq:Rdotestim2}) on $\Raveoint(\TML)$ from the proof regarding $\TML$.
		For $m=1$, notice that $\Ravebo{\TM} = \Rave(\TM)$, hence we can use (\ref{eq:Reff1exact}).
		We thus obtain
		\begin{align*}
			\Raveb{\TMd} & \leq \binom{d}{1} \frac{1}{M^{d-1}} \frac{M}{ 12  } 
						  +  \binom{d}{2} \frac{1}{M^{d-2}} \l[\frac{1}{2\pi}\log M + 1\r]
						  +  \sum_{m=3}^d \binom{d}{m} \frac{1}{M^{d-m}} \frac{4}{m} \\
					  &   \leq   \frac{4}{M^{d}} \sum_{m=1}^d \binom{d}{m} M^m \frac{1}{m} 
					      +  \frac{d}{4 M^{d-2}} \l[ \frac{1}{3} + \frac{(d-1) \log M}{\pi} \r].
		\end{align*}
		After noting that 
		\begin{align*}
			\frac{4}{M^{d}} \sum_{m=1}^d \binom{d}{m} M^m \frac{1}{m} 
				&\leq \frac{4}{M^{d}} \sum_{m=1}^d \binom{d}{m} M^m \frac{2}{m+1} \\
				& \leq \frac{8}{M^{d+1}} \sum_{m=1}^d \binom{d+1}{m+1}  \frac{M^{m+1}}{d+1} \\
				& \leq \frac{8}{d+1} \frac{1}{M^{d+1}} \sum_{n=0}^{d+1} \binom{d+1}{n}  {M^{n}} \\
				&= \frac{8}{d+1}\l(1 + \frac{1}{M} \r)^{d+1},
		\end{align*}
		the thesis follows immediately.
		\hfill$\square$	

\subsection{Analysis for the hypercube $\Hyp{d}$}\label{sect:hypercube} 
		The eigenvalues\footnote{ Note that these eigenvalues cannot be computed using (\ref{eq:eigenvalues}) with $M=2$ because $\Hyp{d}$ is a degenerate case of $T_{2^d}$.}	 of the hypercube $\Hyp{d}$ are
	$ \lambda_{m} = 2 m$ for $m\in \fromto{0}{d}$,  where the eigenvalue 
	$\lambda_m$ has multiplicity $\binom{d}{m} = \frac{d!}{m!(d-m)!}$.
	%
	We thus obtain that
	\begin{align*}
		\Raveb{\Hyp{d}} = \frac{1}{2^d} \sum_{m=1}^d \frac{1}{2m} \binom{d}{m}.
	\end{align*}

	\smallskip
	\noindent
	{\it Proof of Theorem~\ref{thm:hypercube}:}
		For the lower bound, we have:
		\begin{align*}
			\Raveb{\Hyp{d}} 
			& \geq \frac{1}{2^{d+1}} \sum_{m=1}^d \frac{1}{m+1} \binom{d}{m} \\
			&= \frac{1}{2^{d+1}} \sum_{m=1}^d \frac{1}{d+1} \binom{d+1}{m+1}
		\end{align*}
		By the change of variables $m' = m+1$ and $d'=d+1$, we compute 
		$\sum_{m=1}^{d} \binom{d+1}{m+1} = 2^{d+1} -d-2$ and conclude that 
		\begin{equation}\label{eq:lower-Hd}
			\Raveb{\Hyp{d}}\ge \left(1-\frac{d+2}{2^{d+1}} \right)\frac1{d+1}.
		\end{equation}
		%
		For the corresponding upper bound we have:
		\begin{align*}
		\Raveb{\Hyp{d}} 
		            &\leq \frac{1}{2^{d+1}} \sum_{m=1}^d \frac{2}{m+1} \binom{d}{m}\\ 
		            &= \frac{1}{2^{d+1}} \sum_{m=1}^d \frac{2}{d+1} \binom{d+1}{m+1} \leq \frac{2}{d+1}.
		\end{align*}
		\hfill$\square$	
		
	\smallskip
	\noindent	{\it Proof of~\eqref{ReffHypAsy}:}
		 In order to prove the asymptotic trend~\eqref{ReffHypAsy},
		from the definition of $\Raveb{\Hyp{d}}$ and using Pascal's rule we compute:
		\begin{align*}
			\Raveb{\Hyp{d}}
				=&	\frac12 \Raveb{\Hyp{d-1}} + \frac1{2^{d+1}} \frac1d \sum_{k=1}^{d} \binom{d}{k}  \\
				=&	\frac12 \Raveb{\Hyp{d-1}} + \frac1{2d}\l(1-\frac1{2^d}\r)
			\end{align*}
			We have thus shown that the sequence $\Raveb{\Hyp{d}}$ can be constructed recursively by 
			the above formula and defining $\Raveb{\Hyp{0}}=0$. 
			This recursion implies that 
			\begin{align*}
				\Raveb{\Hyp{d}}  	
							=&	\sum_{i=1}^d \frac1{2^{d-i}} \frac1{2i} \l(1-\frac1{2^i}\r)\\
							=&	\sum_{i=1}^d \frac1{2^{d+1}}\frac{2^i-1}{i}
			\end{align*}
			Consequently,
$\displaystyle				\Raveb{\Hyp{d}}\le\frac1{2^{d+1}}\sum_{i=1}^d  \frac{2^i}{i}$
			and we claim that 
			\begin{equation}\label{eq:upper-limit}
				\lim_{d\to+\infty} \frac{\frac1{2^{d+1}}\sum_{i=1}^d  \frac{2^i}{i}}{\frac1d}=1.
			\end{equation}
			This fact can be shown true as follows. Let $a_d=\frac{d}{2^{d+1}}\sum_{i=1}^d\frac{2^i}i.$ 
			Then, it is immediate to verify that $a_d$ satisfies the following recursion
			$$ \begin{cases}
								a_0=0\\
								a_{d+1}=\frac12 \l(1+\frac1d\r)a_{d}+\frac12 \quad \text{for $d\ge0$}
			\end{cases} $$
			and --by induction-- that 
			if $d\ge3$, then $a_d>1$, and
			if $d\ge5$, then $a_{d+1}<a_{d}$.
			Then, $a_d$ must have a finite limit $\ell\ge 1$. Also,
			note that
			$$ a_{d+1}=\frac12 \l(1+\frac1d\r)a_{d}+\frac12 \le \frac12 a_d + \frac43 \frac1d +\frac12.$$
			By taking the limit on both sides of the inequality, we obtain that 
			$\ell\le 1.$ 
			Finally, the desired~\eqref{ReffHypAsy} follows by combining Equations~\eqref{eq:lower-Hd} and~\eqref{eq:upper-limit}. 		\hfill$\square$	

\section{Conclusion}\label{sec:concl}
	The average effective resistance of a graph is an important performance index in several problems of distributed control and estimation, where toroidal grid graphs are exemplary $d$-dimensional graphs.
In these graphs, the asymptotical dependence of the average effective resistance on the network size is well-known, but limited information was available about the constants involved in such relations and about the dependence on the dimension $d$.
	
	We have expressed the average effective resistance of a graph in term of a sum of the inverse Laplacian eigenvalues and found new estimates of this quantity: these estimates are key to our refined asymptotic analysis. For bidimensional toroidal grids, we have identified the proportionality constant of the leading term and we have studied the case when the grid sides have unequal lengths. In grids with $d \geq 3$ and equal side lengths, we conjectured that the average effective resistance is inversely proportional to the dimension $d$. This conjecture is supported by numerical evidences and by several partial results.
	
	Our results have been derived for toroidal grids, but we believe that they provide more general insights about the role of graph dimension is network estimation problems. Indeed, scaling properties deduced on toroidal grid graphs can typically be extended, with due care, to less structured graphs: works in this direction include~\cite{PB-JPH:07,PB-JPH:09,EL-FG-SZ:13,EL-SZ:12}. We envisage that our results on high-dimensional graphs can undergo similar extensions and thus cover more realistic networks in engineering and social sciences.
%
%
%
%
%
\bibliographystyle{plain}

\end{document}